\documentclass[12pt, reqno]{amsart}
\usepackage{amsmath}
\usepackage{amssymb}
\usepackage{amsthm}
\usepackage{enumerate}
\usepackage{xcolor}
\usepackage{mathrsfs}
\usepackage[abbrev]{amsrefs}
\usepackage[utf8]{inputenc}
\usepackage{enumitem}
\usepackage{bm}

\newtheorem{thm}{}[section]
\newtheorem{theorem}[thm]{Theorem}
\newtheorem{corollary}[thm]{Corollary}
\newtheorem{lemma}[thm]{Lemma}
\newtheorem{proposition}[thm]{Proposition}

\theoremstyle{definition}
\newtheorem{definition}[thm]{Definition}
\theoremstyle{remark}

\newtheorem{remark}[thm]{Remark}

\newtheorem{question}[thm]{Question}

\numberwithin{equation}{section}
\allowdisplaybreaks

\newcommand{\XB}{\ensuremath{\mathcal{X}}}
\newcommand{\Uf}{\ensuremath{\mathcal{U}}}
\newcommand{\Ft}{\ensuremath{\mathcal{F}}}
\newcommand{\Nt}{\ensuremath{\mathcal{N}}}
\newcommand{\Mt}{\ensuremath{\mathcal{M}}}
\newcommand{\Ts}{\ensuremath{\mathcal{T}}}
\newcommand{\TT}{\ensuremath{\mathbb{T}}}

\newcommand{\tb}{\ensuremath{\bm{t}}}
\newcommand{\bb}{\ensuremath{\bm{b}}}
\newcommand{\vv}{\ensuremath{\bm{v}}}
\newcommand{\xx}{\ensuremath{\bm{x}}}
\newcommand{\yy}{\ensuremath{\bm{y}}}
\newcommand{\ee}{\ensuremath{\bm{e}}}
\newcommand{\uu}{\ensuremath{\bm{u}}}
\newcommand{\ww}{\ensuremath{\bm{w}}}
\newcommand{\zz}{\ensuremath{\bm{z}}}
\newcommand{\XX}{\ensuremath{\mathbb{X}}}
\newcommand{\YY}{\ensuremath{\mathbb{Y}}}
\newcommand{\BB}{\ensuremath{\mathcal{B}}}
\newcommand{\Id}{\ensuremath{\mathrm{Id}}}
\newcommand{\NN}{\ensuremath{\mathbb{N}}}
\newcommand{\ZZ}{\ensuremath{\mathbb{Z}}}

\newcommand{\FF}{\ensuremath{\mathbb{F}}}
\newcommand{\UU}{\ensuremath{\mathbb{U}}}
\newcommand{\VV}{\ensuremath{\mathbb{V}}}
\newcommand{\WW}{\ensuremath{\mathbb{W}}}
\newcommand{\LL}{\ensuremath{\mathbb{L}}}
\newcommand{\EE}{\mathcal{E}}

\newcommand{\Ker}{\operatorname{Ker}}
\newcommand{\supp}{\operatorname{supp}}

\newcommand{\codim}{\operatorname{codim}}

\AtBeginDocument{
\def\MR#1{}
}

\begin{document}

\title[Uniqueness of unconditional basis of $H_p(\TT)\oplus\ell_{2}$ for $0<p<1$]{Uniqueness of unconditional basis of $\bm{H_p(\TT)\oplus\ell_{2}}$ and $\bm{H_p(\TT)\oplus \Ts^{(2)}}$ for $\bm{0<p<1}$}
\author[F. Albiac]{Fernando Albiac}
\address{Department of Mathematics, Statistics and Computer Sciences, and InaMat$^2$\\ Universidad P\'ublica de Navarra\\
Pamplona 31006\\ Spain}
\email{fernando.albiac@unavarra.es}

\author[J. L. Ansorena]{Jos\'e L. Ansorena}
\address{Department of Mathematics and Computer Sciences\\
Universidad de La Rioja\\
Logro\~no 26004\\ Spain}
\email{joseluis.ansorena@unirioja.es}

\author[P. Wojtaszczyk]{Przemys{\l}aw Wojtaszczyk}
\address{Institute of Mathematics of the Polish Academy of Sciences\\
00-656 Warszawa\\
ul.\ \'Sniadeckich 8\\
Poland}
\email{wojtaszczyk@impan.pl}



\begin{abstract}
Our goal in this paper is to advance the state of the art of the topic of uniqueness of unconditional basis. To that end we establish general conditions on a pair $(\XX, \YY)$ formed by a quasi-Banach space $\XX$ and a Banach space $\YY$ which guarantee that every unconditional basis of their direct sum $\XX\oplus\YY$ splits into unconditional bases of each summand. As application of our methods we obtain that, among others, the spaces $H_p(\TT^d) \oplus\Ts^{(2)}$ and $H_p(\TT^d)\oplus\ell_2$, for $p\in(0,1)$ and $d\in\NN$, have a unique unconditional basis (up to equivalence and permutation).
\end{abstract}

\thanks{F. Albiac acknowledges the support of the Spanish Ministry for Economy and Competitivity under Grant MTM2016-76808-P and the Spanish Ministry for Science and Innovation under Grant PID2019-107701GB-I00 for \emph{Operators, lattices, and structure of Banach spaces}. F. Albiac and J.~L. Ansorena acknowledge the support of the Spanish Ministry for Science, Innovation, and Universities under Grant PGC2018-095366-B-I00 for \emph{An\'alisis Vectorial, Multilineal y Aproximaci\'on}. P. Wojtaszczyk was supported by National Science Centre, Poland grant UMO-2016/21/B/ST1/00241.}

\maketitle

\section{Introduction and background}
\noindent
A relevant topic in Banach space theory from a structural point of view is to determine whether a given space has an unconditional basis and, in the case it does, to know whether this is the unique unconditional basis of the space. Recall that a quasi-Banach space (in particular a Banach space) $\XX$ with a semi-normalized unconditional basis $\BB$ is said to have a \emph{unique unconditional basis} if every semi-normalized unconditional basis of $\XX$ is equivalent to $\BB$, in which case $\BB$ is equivalent to all its permutations, i.e., it is symmetric.

For a Banach space it is rather unusual to have a unique unconditional basis. Indeed, it is well-known that $\ell_{2}$ has a unique unconditional basis \cite{KotheToeplitz1934}, and a classical result of Lindenstrauss and Pe{\l}czy{\'n}ski \cite{LinPel1968} establishes that $\ell_{1}$ and $c_{0}$ also have a unique unconditional basis. Lindenstrauss and Zippin \cite{LinZip1969} completed the picture by showing that those three are the only Banach spaces with this property.

For Banach spaces without a symmetric bases (or Banach spaces that we do not know whether they have a symmetric basis or not) it is more natural to consider the question of uniqueness of unconditional basis up to (equivalence and) permutation, UTAP for short. We say that $\XX$ has a UTAP unconditional basis $\BB$ if every semi-normalized unconditional basis of $\XX$ is equivalent to a permutation of $\BB$. Of course, if $\XX$ has a symmetric basis, the notions of uniqueness of unconditional basis and uniqueness of unconditional basis up to equivalence and permutation coincide. The first movers in this direction of research were Edelstein and Wojtaszczyk, who proved that finite direct sums of $c_0$, $\ell_1$ and $\ell_2$ have a UTAP unconditional basis \cite{EdelWoj1976}. Bourgain et al.\ embarked on a comprehensive study aimed at classifying those Banach spaces with unique unconditional basis up to permutation that culminated in 1985 with their \emph{Memoir} \cite{BCLT1985}. They showed that the spaces $c_{0}(\ell_{1})$, $c_{0}(\ell_{2})$, $\ell_{1}(c_{0})$, $\ell_{1}(\ell_{2})$ and their complemented subspaces with an unconditional basis all have a UTAP unconditional basis, while $\ell_{2}(\ell_{1})$ and $\ell_{2}(c_{0})$ do not. However, the hopes of attaining a satisfactory classification were shattered when they found a nonclassical Banach space, namely the $2$-convexification $\Ts^{(2)}$ of Tsirelson's space $\Ts$ having a UTAP unconditional basis. Other significant advances in the theory were carried out by Casazza and Kalton \cites{CasKal1998,CasKal1999}, who proved that Tsirelson's space $\Ts$, certain Nakano spaces close either to $\ell_1$ or $c_0$, certain complemented subspaces of Orlicz sequence spaces $\ell_F$, where $F$ is a convex Orlicz function close either to the function $t\mapsto t^2$ or to the identity map, and certain infinite $\ell_1$-products and $c_0$-products of spaces with a UTAP unconditional basis have a UTAP unconditional basis. The techniques they developed provided also a new approach to the uniqueness UTAP of unconditional basis in the spaces $c_0(\ell_1)$ and $\ell_1(c_0)$.

Each of these examples of Banach spaces with a UTAP unconditional basis follows one of these mutually exclusive patterns:
\begin{enumerate}[label={\textbf{(P.\arabic*)}}, leftmargin=*,widest=4]
\item\label{closec0} The space is close to $c_0$;
\item\label{closel2} the space is close to $\ell_2$;
\item\label{closel1} the space is close to $\ell_1$;
\item\label{mixed} the space is a finite or infinite direct sum of spaces, each of which follows some of the previous patterns. Moreover, in the case of infinite direct sums, the way in which we sum should also follow one of those patterns.
\end{enumerate}

From the comprehensive approach of Functional Analysis, Banach spaces are a particular case of quasi-Banach spaces, hence it seems only natural to transfer the problem of uniqueness of unconditional basis to this more general setting. As we will illustrate below, the situation for quasi-Banach spaces which are not Banach spaces is quite different. Kalton showed that a wide class of non-locally convex Orlicz sequence spaces, including the $\ell_{p}$ spaces for $0<p<1$, have a unique unconditional basis \cite{Kalton1977}. This topic was given continuity later on in a series of papers, amongst which we mention \cites{Ortynski1981,Leranoz1992,Woj1997,AKL2004,AlbiacLeranoz2008,AlbiacLeranoz2011,AlbiacLeranoz2011b,AlbiacAnsorena2020}. In particular, a wide class of non-locally convex Lorentz sequence spaces were proved to have a unique unconditional basis. As for the UTAP unconditional basis problem, important known results include the cases of finite direct sums of $\ell_p$ spaces for $p\in (0,1]\cup\{2,\infty\}$ (we replace $\ell_\infty$ with $c_0$ is $p=\infty$), the mixed-norm spaces $\ell_p(\ell_2)$, $\ell_p(\ell_1)$, $\ell_1(\ell_p)$, $c_0(\ell_p)$, $\ell_p(c_0)$ for $0<p<1$, and the Hardy spaces $H_{p}(\TT^d)$ for $0<p<1$ and $d\in\NN$. These examples exhibit a pattern which generalizes \ref{closel1}, namely:

\begin{enumerate}[label={\textbf{(P.\arabic*)}},leftmargin=*,widest=4,resume]

\item\label{BEclosel1} the Banach envelope of the space is close to $\ell_1$.

\end{enumerate}
Thus, for quasi-Banach spaces that follow \ref{BEclosel1}, the uniqueness of unconditional basis seems to be the norm rather than the exception.

Pulling the thread of pattern \ref{mixed} suggests the following question.
\begin{question}\label{QXY}
Let $\XX$ and $\YY$ be quasi-Banach spaces with a UTAP unconditional basis. Does $\XX\oplus\YY$ have a UTAP unconditional basis?
\end{question}

Since the methods used in the cited results on uniqueness of unconditional basis depend on the space, finding a positive general answer to Question~\ref{QXY} seems unlikely and remote. The authors of \cite{AlbiacAnsorena2020b} addressed this problem and proved that if $\XX$ and $\YY$ are quasi-Banach spaces with a UTAP unconditional basis falling either  into patterns \ref{BEclosel1} or \ref{closec0}, then $\XX\oplus\YY$ has a UTAP unconditional basis.

The expected way to take the subject further is to study what happens when we consider the direct sum of a quasi-Banach space having a UTAP unconditional basis which additionally falls either into patterns \ref{BEclosel1} or \ref{closec0} (or a combination of both), with a Banach space having a UTAP conditional basis which additionally follows pattern \ref{closel2}. In particular, the following question arises:
\begin{question}\label{question:gluel2}
Does $\XX\oplus\ell_2$ or $\XX\oplus\Ts^{(2)}$ have a UTAP unconditional basis provided that the quasi-Banach space $\XX$ does?
\end{question}

In this paper we address Question~\ref{question:gluel2} in the case when $\XX$ is non-locally convex, and provide a positive answer for a wide class of spaces that includes the Hardy spaces $H_p(\TT^d)$ for $0<p<1$ and $d\in\NN$. We will derive our results from a general splitting principle for unconditional bases that we will present in Section~\ref{sect:main} followed by some applications. Before, in Sections~\ref{sect:split} and \ref{sect:SP}, we will develop the necessary machinery that will sustain our discussion. Specifically, in Section~\ref{sect:split} we set up a technique for splitting unconditional bases of a direct sum of a non-locally convex quasi-Banach space with a Banach space. In Section~\ref{sect:SP} we exploit a method from \cite{Ortynski1979} for splitting complemented subspaces, which combined with the notion of subprojectivity yields a sufficient condition for an operator between quasi-Banach spaces to be strictly singular.

\subsection{Terminology}
We use standard terminology and notation in Banach space theory as can be found, e.g., in \cites{AlbiacKalton2016}. Most of our results, however, will be established in the general setting of quasi-Banach spaces; the unfamiliar reader will find general information about quasi-Banach spaces in \cite{KPR1984}. We next gather the notation that is more heavily used.

A \emph{quasi-norm} on a vector space over the real or complex field $\FF$, is a map $\Vert \cdot\Vert\colon \XX\to [0,\infty)$ satisfying
\begin{enumerate}[label={(Q.\alph*)}]
\item\label{it:H} $\Vert f\Vert = 0$ if and only if $f=0$;
\item\label{it:Hom} $\Vert \alpha f\Vert =|\alpha| \Vert f\Vert$ for $\alpha\in \FF$ and $f\in \XX$; and
\item\label{it:modconcavity} there is a constant $\kappa\ge 1$ so that for all $f$ and $g$ in $\XX$ we have
$
\Vert f +g\Vert \le \kappa (\Vert f\Vert +\Vert g\Vert).
$
\end{enumerate}
If it is possible to take $\kappa=1$ we obtain a norm. More generally, given $0<p\le 1$, a $p$-norm is a map $\Vert \cdot\Vert\colon \XX\to [0,\infty)$ satisfying \ref{it:H}, \ref{it:Hom} and
\begin{enumerate}[resume*]
\item\label{it:pconvex}for all $f$ and $g$ in $\XX$ we have
$
\Vert f+g\Vert^p \le \Vert f\Vert^p +\Vert g\Vert^p.
$
\end{enumerate}
Any $p$-norm is a quasi-norm. Conversely, by Aoki-Rolewicz theorem \cites{Aoki1942, Rolewicz1957} any quasi-norm is $p$-convex for some $p\in(0,1]$, i.e., there is a constant $C\ge 1$ such that
\[
\left\Vert \sum_{j=1}^n f_j \right \Vert^ p \le C \sum_{j=1}^n \Vert f_j\Vert^p, \quad n\in\NN, \, f_j\in\XX.
\]
Thus, any quasi-Banach space can be endowed with an equivalent $p$-norm.

A quasi-norm defines a metrizable vector topology on $\XX$ whose base of neighborhoods of zero is given by sets of the form $\{x\in \XX \colon \Vert x\Vert<1/n\}$, $n\in \NN$. If such topology is complete we say that $(\XX, \Vert\cdot\Vert)$ is a \emph{quasi-Banach space}. A \emph{$p$-Banach space}, $0<p\le 1$, is a quasi-Banach space endowed with a $p$-norm.

The symbol $\XX\simeq\YY$ means that the quasi-Banach spaces $\XX$ and $\YY$ are isomorphic. The \emph{Banach envelope} of a quasi-Banach space $\XX$ consists of a Banach space $\widehat{\XX}$ together with a linear contraction $J_\XX\colon\XX \to \widehat{\XX}$ satisfying the following universal property: for every Banach space $\YY$ and every linear contraction $T\colon\XX \to\YY$ there is a unique linear contraction $\widehat{T}\colon \widehat{\XX}\to \YY$ such that $\widehat{T}\circ J_\XX=T$. Given a basis $\BB$ in $\XX$ we put $\widehat{\BB}:=J_\XX(\BB)$ for the so-called \emph{basis envelope} of $\BB$. We say that a Banach space $\YY$ is isomorphic to the Banach envelope of $\XX$ via the map $J\colon\XX\to\YY$ if the associated map $\widehat{J}\colon\widehat{\XX}\to\YY$ is an isomorphism. For background on envelopes of spaces and bases we refer to \cite{AABW2019}*{\S9}.

A family $(\xx_n)_{n \in \Nt}$ in a quasi-Banach space $\XX$ is said to be \emph{semi-normalized} if
\[
\textstyle
0<\inf_{n\in\Nt} \Vert \xx_n\Vert\le \sup_{n\in\Nt} \Vert \xx_n\Vert<\infty.
\]
The closed linear span of a subset $V$ of $\XX$ will be denoted by $[V]$. A sequence $\XB=(\xx_n)_{n=1}^\infty$ in $\XX$ is said to be a \emph{Schauder basic sequence} if for every $f\in[\xx_n \colon n \in \NN]$ there is a unique family $(a_n)_{n=1}^\infty$ in $\FF$ such that the series $\sum_{n=1}^\infty a_n \, \xx_n$ converges to $f$. If, moreover, $[\xx_n \colon n \in \NN]=\XX$, we say that $\XB$ in $\XX$ is a \emph{Schauder basis} of $\XX$. If $\XB$ is a basis of $\XX$ the functionals $(\xx_n^*)_{n=1}^\infty$ in $\XX^*$ defined by $\xx_n^*(f)=a_n$ whenever $f=\sum_{n=1}^\infty a_n \, \xx_n$ are called the biorthogonal functionals of $\XB$. The support of $f\in\XX$ with respect to the basis $\XB$ is the set
\[
\supp(f)=\{n\in\NN \colon \xx_n^*(f)\not=0\}.
\]
A \emph{block basic sequence} with respect to $\XB$ is a sequence $(\yy_k)_{k=1}^\infty$ in $\XX\setminus\{0\}$ such that
\[
\max(\supp(\yy_k))<\min(\supp(\yy_{k+1})), \quad k\in\NN.
\]

It will be convenient to index unconditional bases with (finite or infinite) countable sets other than $\NN$. A countable family $\BB=(\xx_n)_{n \in \Nt}$ in $\XX$ is an \emph{unconditional basic sequence} if for every $f\in[\xx_n \colon n \in \Nt]$ there is a unique family $(a_n)_{n \in \Nt}$ in $\FF$ such that the series $\sum_{n \in \Nt} a_n \, \xx_n$ converges unconditionally to $f$. If we additionally have $[\xx_n \colon n \in \Nt]=\XX$ then $\BB$ is an \emph{unconditional basis} of $\XX$.

A \emph{sequence space} on a countable set $\Nt$ will be a quasi-Banach lattice on $\Nt$ for which the unit vector system $(\ee_j)_{j\in\Nt}$ defined by $\ee_j=(\delta_{i,j})_{i\in\Nt}$, where $\delta_{i,j}=1$ if $i=j$ and $\delta_{i,j}=0$ otherwise, is an unconditional basis. We will denote by $\EE[\LL]$ the unit vector system of a sequence space $\LL$.

If $\BB$ is an unconditional basis, under a suitable renorming of the space we have
\[
\left\Vert \sum_{n \in \Nt} a_n \, \xx_n\right\Vert \le \left\Vert \sum_{n \in \Nt} b_n \, \xx_n\right\Vert
\]
provided that the vectors $\sum_{n \in \Nt} a_n \, \xx_n$ and $\sum_{n \in \Nt} b_n \, \xx_n$ of $\XX$ satisfy $|a_n|\le |b_n|$ for all $n\in\Nt$. Hence an unconditional basis induces a lattice structure on $\XX$ via the identification of the vectors with the coefficients of their expansions relative to the basis, so that $\XX$ is lattice isomorphic to a sequence space. Because of that, we will say that an unconditional basis enjoys a certain property about lattices when its associated quasi-Banach lattice does. A quasi-Banach lattice $\LL$ is said to be \emph{$q$-convex} (resp., \emph{$q$-concave}), where $0<q\le \infty$, if there is a constant $C>0$ such that for any $m\in\NN$ and $(f_j)_{j=1}^m$ in $\LL$ we have $\Vert f\Vert \le C N$ (resp.\ $N\le C\Vert f\Vert$), where
\begin{equation*}
f=\left(\sum_{j=1}^{m}\vert f_{j}\vert^{q}\right)^{1/q} \text{ and } N=\left(\sum_{j=1}^{m}\Vert f_{j}\Vert^{q}\right)^{1/q}.
\end{equation*}
The general procedure to define the element $f\in\LL$ is described in \cite{LinTza1977II}*{pp.\ 40-41}. However, when the lattice structure on $\LL$ is induced by an unconditional basis $(\xx_n)_{n\in\Nt}$, if
$f_{j}=\sum_{n\in\Nt} a_{j,n} \xx_n$ for $1\le j \le m$, the element $f$ takes the more workable form
\[
\left(\sum_{j=1}^{m}\vert f_{j}\vert^{q}\right)^{1/q} = \sum_{n\in\Nt} \left( \sum_{j=1}^m |a_{j,n}|^q\right)^{1/q} \xx_n.
\]
Related to lattice convexity and lattice concavity are the notions of upper and lower lattice estimates. We say that $\LL$ satisfies an \emph{upper} (resp.\ \emph{lower}) \emph{$q$-estimate} if the above convexity (resp.\ concavity) inequalities hold in the case when $(f_j)_{j=1}^m$ is pairwise disjointly supported. Note that, in this case, $|f|=|\sum_{j=1}^m f_j|$ and so $\Vert f\Vert=\Vert \sum_{j=1}^m f_j\Vert$.

If a quasi-Banach lattice is locally convex as a quasi-Banach space, then it is $1$-convex as a quasi-Banach lattice. However, despite the fact that every quasi-Banach space is $q$-convex for some $0<q\le 1$, there exist quasi-Banach lattices that are not $q$-convex for any $q$. Kalton defined in \cite{Kalton1984} a quasi-Banach lattice $\LL$ as being \emph{$L$-convex} if there is $\varepsilon>0$ so that whenever $f$ and $(f_i)_{i=1}^k$ in $\LL$ satisfy $0\le f_i\le f$ for every $i=1$, \dots, $k$, and $(1-\varepsilon)kf\ge \sum_{i=1}^k f_i$ then $\varepsilon \Vert f \Vert \le \max_{1\le i \le k} \Vert f_i\Vert$. He showed that a quasi-Banach lattice is $L$-convex if and only if it is $q$-convex for some $q>0$. Most quasi-Banach lattices occurring naturally in analysis are $L$-convex. Thus, a quasi-Banach space is said to be \emph{natural} if it is a subspace of an $L$-convex quasi-Banach lattice. Note that, in particular, any Banach space is a natural quasi-Banach space.

The property of $L$-convexity (or local convexity) allows to obtain a tight connection between upper (resp.\ lower) estimates and convexity (resp.\ concavity): if an L-convex lattice $\LL$ satisfies an upper (resp.\ lower) $q$-estimate, then it is a $r$-convex lattice for $0<r<q$ (resp.\ $r$-concave lattice for $q<r<\infty$) (see \cite{Kalton1984}*{Theorem 1.2} and \cite{LinTza1977II}*{Theorem 1.f.7}). Thus, the set of indices $r$ for which $\LL$ is $r$-convex (resp.\ concave) is an interval with lower endpoint $0$ (resp.\ upper endpoint $\infty$).

Suppose that $\BB_x=(\xx_n)_{n \in \Nt}$ and $\BB_y=(\yy_n)_{n \in \Nt}$ are (countable) families of vectors (indexed by the same set $\Nt$) in quasi-Banach spaces $\XX$ and $\YY$, respectively. We say that $\BB_x$ \emph{dominates} $\BB_y$ if there is a bounded linear map $T\colon[\BB_x]\to \YY$ with $T(\xx_n)=\yy_n$ for all $n \in \Nt$. If $\BB_x$ both dominates and it is dominated by $\BB_y$ we say that $\BB_x$ and $\BB_y$ are \emph{equivalent}. Note that an infinite unconditional basis $\BB$ satisfies an upper (resp.\ lower) $p$-estimate if and only if any semi-normalized sequence disjointly supported sequence with respect to $\BB$ is dominated (resp.\ dominates) the unit vector system $\EE[\ell_p]$ of $\ell_p$.

We say that $\BB_x$ is \emph{permutatively equivalent} to a family $\BB_y=(\yy_m)_{m\in \Mt}$ in $\YY$ if there is a bijection $\pi\colon \Nt\to \Mt$ such that $\BB_x$ and $(\yy_{\pi(n)})_{n \in \Nt}$ are equivalent.

Given families $\Lambda_i=(\lambda_{i,j})_{i\in J_i}$ for $i\in I$, we denote by
$\sqcup_{i\in I} \Lambda_i$ its disjoint union, i.e,
\[
\sqcup_{i\in I} \Lambda_i
=(\lambda_{i,j})_{(i,j)\in \cup_{i\in I} \{i\} \times J_i }.
\]

Let $(\XX_i)_{i\in I}$ be a finite collection of (possibly repeated) quasi-Banach spaces. The Cartesian product $\bigoplus_{i\in I}\XX_i$ equipped with the quasi-norm
\[
\textstyle
\left\Vert (\xx_i)_{i\in I}\right\Vert=\sup_{i\in I} \Vert \xx_i\Vert,\quad \xx_i\in\XX_i
\]
is a quasi-Banach space. For $i\in F$ let $L_i\colon \XX_i\to \XX$ be the canonical ``inclusion'' map. Suppose that $\BB_i=(\xx_{i,n})_{j\in \Nt_i}$ is an unconditional basis of $\XX_i$ for each $i\in F$. Then the sequence
\[
\textstyle
\bigoplus_{i\in I} \BB_i :
\sqcup_{i\in F} L_i(\BB_i)
=(L_i(\xx_{i,j}))_{(i,j)\in \cup_{i\in F} \{i\} \times J_i }
\]
is an unconditional basis of $\bigoplus_{i\in I} \XX_i$.

A subspace $\YY$ of a Banach space $\XX$ is said to be \emph{complemented} in $\XX$ if there is a projection $P\colon \XX\to\XX$ with $P(\XX)=\YY$, in which case we say that $\Ker(P)$ is a \emph{complement} of $\YY$ in $\XX$. If $\YY^c$ is a complement of $\YY$ in $\XX$, then $\XX\simeq \YY\oplus \YY^c$ and $\YY^c\simeq\YY/\XX$. This yields a well-known and useful lemma.

\begin{lemma}\label{lem:UC}
Let $\YY$ be a complemented subspace of a quasi-Banach space $\XX$. Suppose that $\UU_1$ and $\UU_2$ are complements of $\YY$ in $\XX$. Then $\UU_1\simeq\UU_2$.
\end{lemma}

An unconditional basic sequence $\BB_y=(\yy_{m})_{m\in \Mt}$ in a quasi-Banach space $\XX$ is said to be \emph{complemented} if its closed linear span $\YY= [\BB_y]$ is a complemented subspace of $\XX$.

A \emph{subbasis} of an unconditional basic sequence $\BB=(\xx_n)_{n \in \Nt}$ is a family $(\xx_n)_{n\in \Mt}$ for some subset $\Mt$ of $\Nt$. Any subbasis of $\BB$ is an unconditional basis sequence which is complemented in $[\bb_n\colon n\in\Nt]$. If $(\Nt_i)_{i\in F}$ is a finite partition of $\Nt$, and we set $\BB_i=(\xx_n)_{n \in \Nt_i}$ for all $i\in F$, we say that $\BB$ \emph{splits into} $(\BB_i)_{i\in F}$. In this case $\BB$ is permutatively equivalent to $\bigoplus_{i\in F} \BB_i$ and so $[\BB]\simeq\bigoplus_{i\in F} [\BB_i]$.

Other more specific notation will be introduced in context when needed.

\section{Splitting unconditional bases of a direct sum\\ of a quasi-Banach space and a Banach space}\label{sect:split}

\noindent This section is geared towards proving Lemma~\ref{lem:split} below. First we introduce the necessary notions as well as some auxiliary results we will use in its proof.

\begin{definition}
We say that a finite family $(\XX_i)_{i\in I}$ of quasi-Banach spaces is \emph{splitting for unconditional bases} if every unconditional basis of $\bigoplus_{i\in I} \XX_i$ splits into basic sequences $(\BB_i)_{i\in I}$ with $[\BB_i]\simeq \XX_i$ for each $i\in I$.
\end{definition}

Being able to split complemented subspaces will also be useful to us.
\begin{definition}
We say that a finite family $(\XX_i)_{i\in I}$ of quasi-Banach spaces is \emph{splitting for complemented subspaces} if for every complemented subspace $\YY$ of $\XX:=\bigoplus_{i\in I} \XX_i$ there are complemented subspaces $\YY_i$ of $\XX_i$ for each $i\in I$ and an automorphism $T$ of $\XX$ such that $T(\YY)=\bigoplus_{i\in I} \YY_i$.
\end{definition}

Since $\ell_1$ is the prototype of the Banach envelope of a quasi-Banach space (for instance $\ell_{1}$ is isometrically isomorphic to the Banach envelope of $\ell_{p}$ for all $0<p<1$ and is also isomorphic to the Banach envelope of the Lorentz sequence space $\ell_{p,q}$ for $0<p<1$ and $0<q\le\infty$) one could conjecture that those Banach spaces ``far from'' $\ell_1$ cannot be Banach envelopes of any non-locally convex quasi-Banach space. Kalton addressed the task of substantiating this guess in \cite{Kalton1986} and showed, for example, that $\ell_2$ is not isomorphic to the Banach envelope of any non-locally convex quasi-Banach space. He also found a non-locally convex quasi-Banach space whose Banach envelope is isomorphic to $c_0$. However, since Kalton also proved that $c_0$ is not the Banach envelope of any natural space, this example can be regarded as somewhat pathological. The following definition brings relief to this feature of a Banach space, which will be exploited thereafter.

\begin{definition}\label{properenveldef}
A Banach space $\XX$ will be said to be a \emph{proper envelope} (respectively, a \emph{proper envelope of  a natural space}) if there is a non-locally convex quasi-Banach space (resp.,\ nonlocally convex natural quasi-Banach space) $\YY$ whose Banach envelope $\widehat{\YY}$ is isomorphic to $\XX$.
\end{definition}

In other words, a Banach space $\XX$ is \emph{not} a proper envelope (resp.\ $\XX$ is not a proper envelope of  natural spaces) if and only if $\XX$ is isomorphic to the Banach envelope of a quasi-Banach space (resp., natural quasi-Banach space) $\YY$ via a map $J\colon \YY\to \XX$ only in the trivial case that $J$ is an isomorphism so that $\YY$ is locally convex.

We will also use a couple of variations of the concept of totally incomparable spaces introduced by Rosenthal in \cite{Rosenthal1969}, whose definitions for quasi-Banach spaces we gather next.

\begin{definition}Two quasi-Banach spaces $\XX$ and $\YY$ will be said to be \emph{totally incomparable} if there is no infinite-dimensional quasi-Banach space isomorphic to both a subspace of $\XX$ and a subspace of $\YY$. If there is no infinite dimensional quasi-Banach space isomorphic to both a complemented subspace of $\XX$ and a subspace of $\YY$, we say that the pair $(\XX,\YY)$ is \emph{semi-complementably incomparable}. If there is no infinite-dimensional quasi-Banach space isomorphic to both a complemented subspace of $\XX$ and a complemented subspace of $\YY$, we say that $\XX$ and $\YY$ are \emph{complementably incomparable}.
\end{definition}

\begin{lemma}\label{lem:CIfromEnv}
Let $\XX$ be a quasi-Banach space and $\UU$ be a Banach space. Suppose that the pair $(\widehat{\XX},\UU)$ is semi-complementably incomparable (resp.\ $\widehat{\XX}$ and $\UU$ are complementably incomparable). Then the pair $(\XX,\UU)$ is semi-complementably incomparable (resp.\ $\XX$ and $\UU$ are complementably incomparable).
\end{lemma}

\begin{proof}
Let $J_\XX\colon\XX\to\widehat{\XX}$ be the envelope map. Suppose that $\YY$ is a complemented subspace of $\XX$ isomorphic to a (complemented) subspace of $\UU$. By \cite{AABW2019}*{Corollary 9.7}, $J_\XX(\YY)$ is complemented in $\widehat{\XX}$, and it is isomorphic to the Banach envelope of $\YY$ via the map $J_\XX|_\YY$. Since $\YY$ is locally convex, $J_\XX|_\YY$ is an isomorphic embedding. Hence, $\YY$ is isomorphic to a complemented subspace of $\widehat{\XX}$ and so $\dim(\YY)<\infty$.
\end{proof}

Given vector spaces $\UU$ and $\VV$ with $\UU\subseteq\VV$ we define de \emph{codimension} of $\UU$ in $\VV$ by
\[
\codim_{\VV}(\UU)=\dim(\VV/\UU).
\]
If $\WW$ is an algebraic complement of $\UU$ in $\VV$, then the canonical linear map from $\WW$ into $\VV/\UU$ is a linear bijection, thus $\codim_{\VV}(\UU)=\dim(\WW)$. It follows from the Hahn-Banach theorem that every Banach space has a great deal of hyperplanes, i.e., closed subspaces of codimension one, and it is well-known that all hyperplanes of a given Banach space are isomorphic. In contrast, there are quasi-Banach spaces with no non-zero functionals and so they contain no hyperplanes; take for instance the space $L_p([0,1])$ for $0<p<1$. Aside from this `pathology', hyperplanes of quasi-Banach spaces behave like those of Banach spaces. That is,  if a quasi-Banach space has hyperplanes then all of them are isomorphic. To evince that this property does not depend on the Hahn-Banach theorem, we write down its proof.

\begin{proposition}\label{lem:finitedcodim}
Let $\VV$ be a closed finite-codimensional subspace of a quasi-Banach space $\XX$. Then,
\begin{enumerate}[label=(\roman*), leftmargin=*,widest=ii]
\item\label{lem:CC} $\VV$ is complemented in $\XX$. In fact, any algebraic complement $\WW$ of $\VV$ is a topological complement. Moreover,
\item\label{lem:CI} if $\UU$ is a closed subspace of $\XX$ with $\codim_{\XX}(\UU)=\codim_{\XX}(\VV)$, then $\UU\simeq\VV$.
\end{enumerate}
\end{proposition}

\begin{proof}
The canonical linear bijection from $\WW$ to $\XX/\VV$ is a topological isomorphism. Hence, if $T\colon\XX/\VV\to\WW$ is its inverse, and $Q\colon \XX\to \XX/\VV$ is the canonical quotient map, $T\circ Q$ is a projection onto $\WW$ whose kernel is $\VV$. Combining the formulas
\begin{align*}
\codim_{\VV}(\UU\cap\VV) & \le \codim_{\XX}(\UU),\\
\codim_{\XX}(\UU\cap\VV)&=\codim_{\VV}(\UU\cap\VV)+\codim_{\XX}(\VV)
\end{align*}
with the ones we get from switching the roles of $\UU$ and $\VV$ yields
\[
n:=\codim_{\UU}(\UU\cap\VV)=\codim_{\VV}(\UU\cap\VV)<\infty.
\]
By \ref{lem:CC}, $\VV\simeq (\UU\cap\VV) \oplus \FF^n\simeq\UU$.
\end{proof}

\begin{lemma}\label{lem:cycle}
Let $\XX$ be a quasi-Banach space. There is $d=d(\XX)\in\NN\cup\{0\}$ such that:
\begin{enumerate}[label=(\roman*), leftmargin=*,widest=ii]
\item\label{lem:cycleup} If $\UU$ and $\VV$ are finite-dimensional quasi-Banach spaces, then $\XX\oplus\UU\simeq \XX\oplus\VV$ if and only if $\dim(\UU)-\dim(\VV)=jd$ for some $j\in\ZZ$.
\item\label{lem:cycledown} If $\UU$ and $\VV$ are two finite-codimensional closed subspaces of $\XX$, then $\UU\simeq \VV$ if and only if $\codim_{\XX}(\UU)-\codim_{\XX}(\VV)=jd$ for some $j\in\ZZ$.
\end{enumerate}
\end{lemma}

\begin{proof}
Set $\XX[n]=\XX\oplus\FF^n$ for $n\in\NN\cup\{0\}$. If the spaces $(\XX[n])_{n=0}^\infty$ are mutually non-isomorphic, then \ref{lem:cycleup} holds with $d=0$. Assume that it is not the case and pick $a\in\NN\cup\{0\}$ minimal with the property that $\XX[a]\simeq\XX[n]$ for some $n>a$. If $a>0$ we would have
\[
\XX[a-1]\oplus \FF\simeq\XX[a]\simeq \XX[n-1]\oplus\FF.
\]
From Proposition~\ref{lem:finitedcodim} we would get $\XX[a-1]\simeq \XX[n-1]$, which contradicts the minimality of $a$. Hence, there is $d\in\NN$ minimal with the property that $\XX\simeq\XX[d]$. Given $m\in\NN\cup\{0\}$ we have $\XX[m]\simeq\XX[m+d]$. Using induction we obtain that $\XX[m]\simeq\XX[m+jd]$ for all $j\in\NN$.

Suppose that $0\le m<n$ are such that $\XX[m]\simeq\XX[n]$. Write $m=id+m'$ and $n=m+jd+d'$ with $i$, $j\in\NN\cup\{0\}$ and $m'$, $d'\in\ZZ\cap[0,d-1]$.
Then,
\begin{align*}
\XX&\simeq \XX[(i+1)d]
\simeq \XX[m]\oplus\FF^{d-m'}
\simeq \XX[(i+j)d+m'+d']\oplus\FF^{d-m'}\\
&\simeq \XX[(i+j+1)d+d']
\simeq \XX[d'].
\end{align*}
The minimality of $d$ gives $d'=0$, so $n-m=jd$. This proves \ref{lem:cycleup}.

Let $\UU$, $\VV$ be as in \ref{lem:cycledown}. Use Proposition~\ref{lem:finitedcodim} to pick complements
$\UU^c$ and $\VV^c$ of $\UU$ and $\VV$, respectively. By Proposition~\ref{lem:finitedcodim}, $\UU\simeq \VV$ if and only if
\[
\XX_1:= \UU \oplus \UU^c\oplus\VV^c \simeq \XX_2:= \VV \oplus \UU^c\oplus\VV^c.
\]
Since $\XX_1\simeq \XX\oplus \VV^c$ and $\XX_2\simeq \XX\oplus \UU^c$, applying \ref{lem:cycleup} yields that $\UU\simeq \VV$ if and only if $\dim(\VV^c)-\dim(\UU^c)\in d\,\ZZ$.
\end{proof}

\begin{remark}
Most Banach spaces $\XX$ are isomorphic to their hyperplanes, or, in the terminology of Lemma~\ref{lem:cycle}, $d(\XX)=1$. In fact, Banach \cite{Banach1932} conjectured that any Banach space should have this property. This question was solved in the negative in \cite{Gowers1994} by Gowers, who exhibited examples of Banach spaces $\XX$ with $d(\XX)=0$. Subsequently, Gowers and Maurey constructed for every $d$ a separable Banach space $\XX$ with $d(\XX)=d$ (see \cite{GowMau}*{Theorem 26 and following remarks}).
\end{remark}

Now we are ready to state and prove the general condition for the pair $(\XX,\UU)$ to be splitting for unconditional bases.
\begin{lemma}\label{lem:split}
Let $\XX$ be a quasi-Banach space and $\UU$ be a Banach space.
Suppose that:
\begin{enumerate}[label=(\roman*),leftmargin=*,widest=iii]
\item\label{c:SCS} $(\XX,\UU)$ is splitting for complemented subspaces;
\item $(\widehat{\XX} , \UU)$ is splitting for unconditional bases;
\item $\widehat{\XX}$ and $\UU$ are complementably incomparable; and
\item\label{c:FNL} either $\UU$ is not a proper envelope, or $\XX$ is natural and $\UU$ is not a proper envelope of natural spaces.
\end{enumerate}
Then the pair $(\XX,\UU)$ is splitting for unconditional bases.
\end{lemma}

\begin{proof}
Let $J_\XX\colon\XX\to\widehat{\XX}$ be the envelope map. By \cite{AABW2019}*{Corollary 9.7}, $\widehat{\XX}\oplus \UU$ is isomorphic to the Banach envelope of $\XX\oplus \UU$ via the map $J=(J_\XX,\Id_\UU)$.

Let $\BB=(\bb_n)_{n\in\Nt}$ be an unconditional basis of $\XX\oplus\UU$. By \cite{AABW2019}*{Proposition 9.9}, $J(\BB)$ is an unconditional basis of $\widehat{\XX}\oplus\UU$. Therefore we can choose a partition $(\Nt_y,\Nt_v)$ of $\Nt$ such that, if we denote $\BB_y=(\bb_n)_{n\in\Nt_y}$ and $\BB_v=(\bb_n)_{n\in\Nt_v}$, $J(\BB_y)$ generates a space isomorphic to $\widehat{\XX}$ and $J(\BB_v)$ generates a space isomorphic to $\UU$. Applying again \cite{AABW2019}*{Corollary 9.7} we obtain that $J([\BB_y])$ is isomorphic to the Banach envelope of $[\BB_y]$ via the map $J|_{[\BB_y]}$, and that $J([\BB_v])$ is isomorphic to the Banach envelope of $[\BB_v]$ via the map $J|_{[\BB_v]}$.

Since $J([\BB_v])\simeq \UU$, the map $J|_{[\BB_v]}$ is an isomorphic embedding. If $\XX$ is natural then $[\BB_v]$ is natural therefore $\BB_v$ generates a space isomorphic to $\UU$.

Since $[\BB_y]$ is complemented in $\XX\oplus\UU$, there exist a complemented subspace $\XX_0$ of $\XX$, a complemented subspace $\UU_0$ of $\UU$, and an automorphism $T$ of $\XX\oplus\UU$ such that $T([\BB_y])=\XX_0\oplus\UU_0$. In particular,
\[
[\BB_y]\simeq \XX_0\oplus\UU_0.
\]

If $\XX^c_0$ is a complement of $\XX_0$ in $\XX$ and $\UU^c_0$ is a complement of $\UU_0$ in $\UU$, then $\XX_0^c\oplus\UU_0^c$ is a complement of $\XX_0\oplus\UU_0$ in $\XX\oplus\UU$. Since $[\BB_u]$ is a complement of $[\BB_y]$ in $\XX\oplus\UU$, $T([\BB_v])$ is a complement of $\XX_0\oplus\YY_0$ in $\XX\oplus \UU$. Applying Lemma~\ref{lem:UC} we obtain
\[
\XX_0^c\oplus\UU_0^c\simeq T([\BB_v]) \simeq [\BB_v] \simeq \UU \simeq \UU_0\oplus \UU_0^c.
\]
$\UU_0$ is a complemented subspace of $\UU$ isomorphic to a complemented subspace of $[\BB_y]$. Since the Banach envelope of $[\BB_y]$ is isomorphic to $\widehat{\XX}$, applying Lemma~\ref{lem:CIfromEnv} yields that $[\BB_y]$ and $\UU$ are complementably incomparable. Thus, $\dim(\UU_0)<\infty$. In turn, $\XX_0^c$ is a complemented subspace of $\XX$ isomorphic to a complemented subspace of $\UU$. Applying again Lemma~\ref{lem:CIfromEnv} yields $\dim(\XX_0^c)<\infty$.

Let $d=d(\UU)\in\NN\cup\{0\}$ be as in Lemma~\ref{lem:cycle}. We infer that there is $j\in\ZZ$ such that
\[
\dim(\XX_0^c)-\dim(\UU_0)=jd.
\]
If $d=0$ or $d\ge 1$, and $j=0$ we have
\[
[\BB_y]\simeq\XX_0\oplus \UU_0\simeq \XX_0\oplus \XX_0^c\simeq \XX.
\]
If $d>0$ and $j>0$ we pick $\Ft\subseteq \Nt_v$ with $|\Ft|=jd$, and we set
\[
\BB_f=(\bb_n)_{n\in \Ft}, \quad
\BB_x=(\bb_n)_{n\in \Nt_y \cup \Ft} \text{, and} \quad
\BB_u=(\bb_n)_{n\in \Nt_v\setminus \Ft}.
\]
By Lemma~\ref{lem:cycle}~\ref{lem:cycledown}, $[\BB_u]\simeq\UU$. Moreover,
\[
[\BB_x]\simeq [\BB_y]\oplus [\BB_f] \simeq \XX_0\oplus\UU_0\oplus[\BB_f]
\simeq\XX_0\oplus\XX_0^c\simeq\XX.
\]
Finally, if $d>0$ and $j<0$ we pick $\Ft\subseteq \Nt_y$ with $|\Ft|=-jd$, and we set
\[
\BB_f=(\bb_n)_{n\in \Ft}, \quad
\BB_x=(\bb_n)_{n\in \Nt_y\setminus \Ft} \text{, and} \quad
\BB_u=(\bb_n)_{n\in \Nt_v\cup \Ft}.
\]
By Lemma~\ref{lem:cycle}~\ref{lem:cycleup}, $[\BB_u]\simeq [\BB_v]\oplus [\BB_f]\simeq\UU\oplus [\BB_f]\simeq\UU$. Moreover,
\[
[\BB_x]\oplus\UU_0
\simeq [\BB_x] \oplus [\BB_f] \oplus \XX_0^c
\simeq [\BB_y] \oplus \XX_0^c
\simeq \XX_0 \oplus \UU_0 \oplus \XX_0^c
\simeq \XX\oplus \UU_0.
\]
By Lemma~\ref{lem:finitedcodim}, $[\BB_x]\simeq\XX$.
\end{proof}

\section{The role of subprojectivity}\label{sect:SP}
\noindent
The following lemma gathers elementary connections between several concepts relevant to this paper.
\begin{lemma}\label{lem:relations}
Consider the following properties involving two quasi-Banach spaces $\XX$ and $\YY$:
\begin{enumerate}[label=(\roman*),leftmargin=*,widest=iii]
\item\label{it:C}Every bounded linear operator from $\XX$ into $\YY$ is compact.
\item\label{it:I} $\XX$ and $\YY$ are totally incomparable.
\item\label{it:SS} Every bounded linear operator from $\XX$ into $\YY$ is strictly singular.
\item\label{it:SCI} $(\XX,\YY)$ is semi-complementably incomparable.
\item\label{it:CI} $\XX$ and $\YY$ are complementably incomparable.
\end{enumerate}
We have \ref{it:C} $\Rightarrow$ \ref{it:SS}, \ref{it:I} $\Rightarrow$ \ref{it:SS}, \ref{it:SS} $\Rightarrow$ \ref{it:SCI}, and \ref{it:SCI} $\Rightarrow$ \ref{it:CI}.
\end{lemma}

\begin{proof}
Only \ref{it:SS} $\Rightarrow$ \ref{it:SCI} deserves to be sketched. Suppose  \ref{it:SCI} does not hold. Then, there are  a projection $P\colon \XX\to\XX$ onto an infinite-dimensional subspace $\UU$, and an isomorphic embedding $T\colon\UU \to \YY$. The operator $T\circ P\colon\XX\to\YY$ is an isomorphic embedding when restricted to $\UU$.
\end{proof}
We will next see that under a mild condition on $\YY$, conditions \ref{it:SS} and \ref{it:SCI} in Lemma~\ref{lem:relations} are in fact equivalent.

\begin{definition}
We say that a quasi-Banach space $\YY$ is \emph{subprojective} if for every infinite dimensional subspace $\VV$ of $\YY$ there is a further subspace $\UU\subseteq\VV$ which is complemented in $\YY$. If, in addition, $\UU$ is isomorphic to one of the members of a given set $\Uf$ of infinite-dimensional quasi-Banach spaces, we say that $\YY$ is $\Uf$-subprojective (or complementably $\Uf$-saturated).
\end{definition}

Loosely speaking, the following lemma tells us that subprojectivity serves as the key that allows to pull-back complemented subspaces via non-strictly singular operators. Although the result is essentially known, for the sake of completeness we include a proof.

\begin{lemma}[cf.\ \cite{OS2015}*{Corollary 2.4}]\label{lem:ComPBlp}
Let $\VV$ be an infinite-dimensional subspace of a quasi-Banach space $\XX$, let $\Uf$ be a family of infinite-dimensional quasi-Banach spaces, and let $\YY$ be a $\Uf$-subprojective quasi-Banach space. Suppose that there is a bounded linear operator $T\colon\XX\to\YY$ such that $T|_\VV$ is an isomorphic embedding. Then there is a subspace of $\VV$ which is complemented in $\XX$ and isomorphic to some member of $\Uf$.
\end{lemma}

\begin{proof}
Passing to a subspace we can assume that $T(\VV)$ is complemented in $\YY$ and isomorphic to $\UU$ for some $\UU\in\Uf$. Let $S\colon T(\YY)\to \YY$ be the inverse operator of $T|_\VV$, and let $P\colon \YY\to T(\VV)$ be such that $P|_{T(\VV)}=\Id_{T(\VV)}$. Since the map $S\circ P\circ T\colon \XX\to\VV$ is the identity on $\VV$, $\VV$ is complemented in $\XX$.
\end{proof}

Banach space old-timers will be surely aware of the fact that a Banach space $\XX$ admits a non-strictly singular operator into $\ell_2$ if and only if $\ell_2$ is isomorphic to a complemented subspace of $\XX$. The following lemma provides an extension of this result.

\begin{proposition}\label{prop:CItoSS}
Let $\XX$ and $\YY$ be quasi-Banach spaces and $\Uf$ be a set of infinite-dimensional quasi-Banach spaces. Suppose that no space in $\Uf$ is isomorphic to a complemented subspace of $\XX$ and that $\YY$ is $\Uf$-subprojective. Then, every operator from $\XX$ into $\YY$ is strictly singular.
\end{proposition}

\begin{proof}
Assume by contradiction that $T\colon \XX\to \YY$ is an isomorphic embedding when restricted to an infinite-dimensional subspace $\VV$ of $\XX$. By Lemma~\ref{lem:ComPBlp}, there is an infinite-dimensional subspace of $\VV$ which is complemented in $\XX$ and isomorphic to a space from $\Uf$.
\end{proof}

Our next result is a straightforward consequence of Proposition~\ref{prop:CItoSS}.
\begin{corollary}\label{cor:CItoSS}
Let $\XX$ and $\YY$ be quasi-Banach spaces. Suppose that $(\XX,\YY)$ is semi-complementably incomparable and that $\YY$ is subprojective. Then every operator from $\XX$ into $\YY$ is strictly singular.
\end{corollary}

\subsection{Subprojective Banach spaces}
Although the definition of subprojectivity makes sense for quasi-Banach spaces, we know no example of a non-locally convex subprojective quasi-Banach space. The result of Stiles \cite{Stiles1972} that the space $\ell_p$ for $0<p<1$ is not subprojective, points out in the direction that no non-locally convex quasi-Banach space can be subprojective. For this reason, in this paper we will keep within bounds of locally convex spaces as far as subprojectivity is concerned. The applications we will obtain will rely on the results on the subject from \cite{OS2015}, were it is proved that subprojectivity is inherited by direct sums of Banach spaces. To be precise, we have the following.

\begin{proposition}[cf.\ \cite{OS2015}*{Proposition 2.2}]\label{prop:SPDS}
Let $I$ be a finite set. Suppose that for each $i\in I$, $\Uf_i$ is a set of infinite-dimensional Banach spaces and $\XX_i$ is a $\Uf_i$-subprojective Banach space. Then the space $\bigoplus_{i\in I} \XX_i$ is $\Uf$-subprojective, where $\Uf=\bigcup_{i\in I} \Uf_i$.
\end{proposition}

We emphasize that the proof of Proposition~\ref{prop:SPDS} does not carry over to quasi-Banach spaces since it is by no means clear whether the sum of two strictly singular operators is strictly singular in general. We refer the reader also to the article \cite{OS2015} for a list of subprojective Banach spaces. This section is aimed at adding to this list the convexifications of Tsirelson's space and their duals. To that end, we first introduce a few lemmas.

We say that a family $(\uu_n)_{n\in\Nt}$ in a quasi-Banach space $\XX$ and a family $(\vv_n)_{n\in\Nt}$ in a quasi-Banach space $\YY$ are \emph{congruent} (in $\XX$ and $\YY$) if there is an isomorphism $T$ of $\XX$ onto $\YY$ with $T(\uu_n)=\vv_n$ for all $n\in\Nt$. Congruence is stronger than equivalence. We introduce it because congruence ensures that if a subspace $\UU$ of $[\uu_n\colon n\in\Nt]$ is complemented in $\XX$, then the corresponding subspace $T(\UU)$ of $[\vv_n\colon n\in\Nt]$ is complemented in $\YY$.

\begin{lemma}[cf.\ \cite{BP1958}*{C2}]\label{lem:BSQB}
Let $I$ be a finite set. For each $i\in I$ let $\XX_i$ be a quasi-Banach space with a Schauder basis $\BB_i$. Suppose that $\XX$ is an infinite-dimensional quasi-Banach space and that $J_i\colon \XX\to \XX_i$ is an isomorphic embedding for each $i\in I$. Then $\XX$ has a basic sequence $\BB$ equivalent to a block basic sequence $\BB_i'$ with respect to $\BB_i$ for all $i\in I$. Moreover, if $\XX_i$ is locally convex, $J_i(\BB)$ and $\BB_i'$ are congruent.
\end{lemma}

\begin{proof}
There is a sequence $(f_n)_{n=1}^\infty$ in $\XX$ with $\sup_n\Vert f_n\Vert<\infty$ and $\inf_{n\not=m} \Vert f_n-f_m\Vert>0$ (see \cite{AlbiacAnsorena2020}*{Lemma 2.8}). Passing to a subsequence, Cantor's classical diagonal argument gives that $(J_i(f_n))_{n=1}^\infty$ converges coordinate-wise with respect $\BB_i$ for all $i\in I$. Set $\yy_n=f_{2n-1}-f_{2n}$ for $n\in\NN$. Since $\BB:=(\yy_n)_{n=1}^\infty$ is semi-normalized and $J_i(\BB)$ is coordinate-wise null with respect $\BB_i$ for all $i\in I$, combining the gliding hump technique with the principle of small perturbations, passing to a further subsequence we obtain that $J_i(\BB)$ is equivalent (congruent if $\XX_i$ is locally convex) to a block basic sequence with respect to $\BB_i$ for all $i\in I$.
\end{proof}

As a by-product, we obtain conditions under which a quasi-Banach space always contains a basic sequence. Recall that the basic sequence problem for quasi-Banach spaces was solved in the negative by Kalton in \cite{Kalton1995}.
\begin{corollary}\label{cor:BSQB} Let $\XX$ be an infinite-dimensional quasi-Banach space.
If $\XX$ embeds in a quasi-Banach space with a Schauder basis, then $\XX$ contains a basic sequence.
\end{corollary}

\begin{lemma}\label{lem:CSProj}
Let $\XX$ be a Banach space with an unconditional basis $\BB=(\xx_j)_{j\in\Nt}$. Let $\YY$ be the subspace of $\XX^*$ spanned by the biorthogonal functionals $\BB^*=(\xx_j^*)_{j\in\Nt}$ of $\BB$. Suppose that there are a partition $(\Nt_n)_{n=1}^\infty$ of $\Nt$ into finite sets and a constant $C$ such that
$
\left\Vert \sum_{n=1}^\infty f_n\right\Vert \le C \left\Vert \sum_{n=1}^\infty g_n\right\Vert
$
whenever $(f_n)_{n=1}^\infty$ and $(g_n)_{n=1}^\infty$ in $\XX$ satisfy $\Vert f_n\Vert \le \Vert g_n\Vert$ and $\supp(f_n)\cup\supp(g_n)\subseteq \Nt_n$ for all $n\in\NN$.
\begin{enumerate}[label=(\roman*),leftmargin=*,widest=ii]
\item\label{it:BBSinX} If $\BB_u=(\uu_n)_{n=1}^\infty$ and $\BB_v=(\vv_n)_{n=1}^\infty$ are semi-normalized sequences in $\XX$ with $\supp(\uu_n)\cup\supp(\vv_n)\subseteq\Nt_n$ for all $n\in\NN$, then $\BB_u$ and $\BB_v$ are equivalent. Moreover, $[\BB_u]$ and $[\BB_v]$ are complemented in $\XX$.
\item\label{it:BBSinY} If $\BB_y^*=(\yy_n^*)_{n=1}^\infty$ and $\BB_w^*=(\ww_n^*)_{n=1}^\infty$ are semi-normalized sequences in $\YY$ with $\supp(\yy_n^*)\cup\supp(\ww_n^*)\subseteq\Nt$ for all $n\in\NN$, then $\BB_y^*$ and $\BB_w^*$ are equivalent. Moreover, $[\BB_y^*]$ and $[\BB_w^*]$ are complemented in $\YY$.
\end{enumerate}
\end{lemma}

\begin{proof}

Consider the following conditions on a pair $(\BB_z,\BB^*_z)$ formed by a sequence $\BB_z=(\zz_n)_{n=1}^\infty$ in $\XX$ and a sequence $\BB^*_z=(\zz^*_n)_{n=1}^\infty$ in $\XX^*$:
\begin{enumerate}[label=($\clubsuit$),leftmargin=*]
\item\label{it:BBSXY} $\zz_n^*(\zz_n)=1$, \quad $\supp(\zz_n)\cup\supp(\zz_n^*)\subseteq\Nt_n$ for all $n\in\NN$,\quad and \newline $k[\BB_z,\BB_z^*]:=\sup_n \max\{\Vert \zz_n \Vert, \Vert \zz_n^* \Vert\}<\infty$.
\end{enumerate}

Under the assumptions in \ref{it:BBSinX} (resp.\ in \ref{it:BBSinY}) there are sequences $\BB_u^*$ and $\BB_v^*$ in $\XX^*$ (resp.\ $\BB_y$ and $\BB_w$ in $\XX$) such that the pairs
$(\BB_u,\BB_u^*)$ and $(\BB_v,\BB_v^*)$ (resp.\ $(\BB_y,\BB_y^*)$ and $(\BB_w,\BB_w^*)$) fulfil \ref{it:BBSXY}. So, it suffices to prove that if $(\BB_u,\BB_u^*)$ and $(\BB_v,\BB_v^*)$ satisfy
\ref{it:BBSXY}, then $\BB_u$ and $\BB_v$ are equivalent, $\BB_u^*$ and $\BB_v^*$ are equivalent, $[\BB_u]$ and $[\BB_u]$ are complemented in $\XX$, and $[\BB_u^*]$ and $[\BB_v^*]$ are complemented in $\XX^*$.

Set $\BB_u^*=(\uu_n^*)_{n=1}^\infty$, $\BB_v^*=(\vv_n^*)_{n=1}^\infty$, and $k=k[\BB_u,\BB_u^*]$. If we denote by $S_A\colon\XX\to\XX$ the coordinate projection on a set $A\subseteq \Nt$ we have
\[
|\uu_n^*(f)| \, \Vert \uu_n\Vert \le k^2 \Vert S_{\Nt_n}(f) \Vert,\quad f\in\XX, \ n\in\NN.
\]
Hence, the linear map $P\colon\XX\to\XX$ given by
\[
P(f)=\sum_{n=1}^\infty \uu_n^*(f) \vv_n, \quad f\in \XX
\]
is well-defined and satisfies $\Vert P\Vert \le k^2 C$. We also have $P(\uu_n)=\vv_n$ for all $n\in\NN$. The dual map $P^*\colon \XX^*\to \XX^*$ is given by
\[
P^*(f^*) (f) =\sum_{n=1}^\infty f^*( \vv_n) \uu_n^*(f), \quad f\in\XX, \ f^*\in\XX^*.
\]
Hence, $P^*(\vv_n^*)=\uu_n^*$ for all $n\in\NN$.

The operators defined when replacing $\BB_v$ with $\BB_u$ and $\BB_u^*$ with $\BB_v^*$ yield the equivalence between $\BB_u$ and $\BB_v$, as well as the equivalence between $\BB_u^*$ and $\BB_v^*$. The operator defined when replacing only $\BB_v$ with $\BB_u$ yields projections from $\XX$ onto $[\BB_u]$ and from $\XX^*$ onto $[\BB_u^*]$. Finally, the operator defined when replacing only $\BB_u^*$ with $\BB_v^*$ yields projections from $\XX$ onto $[\BB_v]$ and from $\XX^*$ onto $[\BB_v^*]$.
\end{proof}

\begin{theorem}\label{thm:AnsoSP}
Let $\XX$ be a Banach space with an unconditional basis $\BB=(\xx_j)_{j=1}^\infty$. Let $\YY$ be the subspace of $\XX^*$ spanned by the biorthogonal functionals $\BB^*=(\xx_j^*)_{j\in\Nt}$ of $\BB$. Suppose that there are a constant $C$ and an increasing sequence $(j_n)_{n=1}^\infty$ in $\NN$ with the following property:
$
\left\Vert \sum_{n=1}^\infty f_n\right\Vert \le C \left\Vert \sum_{n=1}^\infty g_n\right\Vert
$
whenever $(f_n)_{n=1}^\infty$ and $(g_n)_{n=1}^\infty$ in $\XX$ and $(k_n)_{n=1}^\infty$ in $\NN$ satisfy $\Vert f_n\Vert \le \Vert g_n\Vert$, $\supp(f_n)\cup\supp(g_n) \subseteq[k_n, k_{n+1}-1]$, and $j_n\le k_n$ for all $n\in\NN$. Then, $\XX$ is $\Uf$-subprojective and $\YY$ is $\Uf^*$-subprojective, where
\begin{align*}
\Uf&=\{[\xx_{k_n} \colon n\in\NN]\colon\quad \forall n\in\NN,\ j_n\le k_n<k_{n+1}\},\\
\Uf^*&=\{[\xx_{k_n}^* \colon n\in\NN]\colon \quad \forall n\in\NN,\ j_n\le k_n<k_{n+1}\}.
\end{align*}
\end{theorem}

\begin{proof}
Pick a subspace $\UU$ of $\XX$ (resp.\ of $\YY$). By Lemma~\ref{lem:BSQB}, passing to a subspace we can suppose that $\UU$ is spanned by a basic sequence, say $(\uu_k)_{k=1}^\infty$, congruent to a block basic sequence $(\vv_k)_{k=1}^\infty$. There is an increasing sequence $(k_n)_{n=1}^\infty$ in $\NN$ such that $j_n\le \min(\supp(\vv_{k_n}))$ for all $n\in\NN$. By Lemma~\ref{lem:CSProj}, $(\vv_{k_n})_{n=1}^\infty$ spans a complemented subspace of $\XX$ (resp.\ of $\YY$) isomorphic to $\VV:=[\xx_{k_n} \colon n\in\NN]$ (resp.\ $\VV:=[\xx_{k_n}^*\colon n\in\NN]$). By congruence  (equivalence is not enough!), $(\uu_{k_n})_{n=1}^\infty$ spans a complemented subspace of $\XX$ (resp.\ of $\YY$) isomorphic to $\VV$.
\end{proof}

It is known that Tsirelson's space $\Ts$ is subprojective \cite{GMBSB2011}*{Proposition 2.4}. This fact can be derived from the properties of the lattice structure on $\Ts$, and this is what we will use to show the subprojectivity of their convexifications and their duals. Given $0<r<\infty$, $\Ts^{(r)}$ denotes de quasi-Banach lattice consisting of all $f\in\FF^\NN$ such that $|f|^r\in\Ts$. Given $1<s\le \infty$, we denote by $\Ts_*^{(s)}$ the dual of the $r$-convexified Tsirelson's space $\Ts^{(r)}$, where $r=s/(s-1)$. With this terminology, $\Ts_*^{(\infty)}$ is the original Tsirelson's space $\Ts^*$. Since $\Ts$ is $1$-convex and $p$-concave for any $p>1$, $\Ts^{(r)}$ is $r$ convex and $p$ concave for any $p>r$. Consequently, $\Ts_*^{(s)}$ is $s$-concave and $p$-convex for any $p<s$.
\begin{theorem}\label{thm:TsirelsonSP}
Let $1\le r<\infty$ and $1<s\le \infty$. Let $(\tb_n)_{n=1}^\infty$ denote the unit vector system of $\Ts^{(r)}$, and $(\tb_n^*)_{n=1}^\infty$ denote the unit vector system of $\Ts_*^{(s)}$. Then $\Ts^{(r)}$ is $\Uf$-subprojective and $\Ts_*^{(s)}$ is $\Uf^*$-subprojective, where,
\begin{align*}
\Uf&=\{[\tb_{k_n} \colon n\in\NN] \colon \quad \forall n\in\NN,\ j_n\le k_n<k_{n+1}\},\\
\Uf^*&=\{[\tb_{k_n}^* \colon n\in\NN]\colon \quad \forall n\in\NN,\ j_n\le k_n<k_{n+1}\},
\end{align*}
and $(j_n)_{n=1}^\infty$ is an arbitrary increasing sequence in $\NN$.
\end{theorem}

\begin{proof}
It is known that the unit vector system of $\Ts$ satisfies for any $(j_n)_{n=1}^\infty$ the assumptions in Theorem~\ref{thm:AnsoSP} (see \cite{CasShura1989}*{Corollary II.5}). Since $r$-convexifications inherits this property, the unit vector system $\Ts^{(r)}$ also does.
\end{proof}

We can also apply Theorem~\ref{thm:AnsoSP} to Nakano spaces. Given a sequence $(p_n)_{n=1}^\infty$ in $[1,\infty)$ we denote by $\ell(p_{n})$ the Banach space built from the modular
\[
(a_n)_{n=1}^\infty\mapsto \sum_{n=1}^\infty |a_n|^{p_n},
\]
and we denote by $h(p_{n})$ the separable part of $\ell(p_{n})$. We have $\ell(p_{n})=h(p_{n})$ if and only if $\sup_n p_n<\infty$. If there is $0<R<1$ such that
\[
\sum_{n=1}^\infty R^{pp_n/|p_n-p|}<\infty,
\]
then $\ell(p_{n})=\ell_p$ up to an equivalent norm. Since $\ell_p$ spaces are particular cases of Nakano spaces, the following generalizes \cite{Pel1960}*{Lemma 2}, which establishes that $\ell_p$ is $\ell_p$-subprojective.

\begin{theorem}[cf.\ \cite{RuizSanchez2018}*{Theorem 5.1}]
Let $(p_n)_{n=1}^\infty$ be a sequence in $[1,\infty)$. Suppose that there exists $\lim_n p_n=p\in[1,\infty]$. Then $h(p_{n})$ is $\ell_p$-subprojective (we replace $\ell_\infty$ with $c_0$ if $p=\infty$).
\end{theorem}

\begin{proof}
Set for each $n\in\NN$
\[
R_n=\max_{k\ge n} \frac{p p_k}{|p-p_k|}.
\]
Pick an increasing sequence $(n_j)_{j=1}^\infty$ in $\NN$ with $\sum_{j=1}^\infty 2^{-R_{n_j}}<\infty$. Any semi-normalized block basic sequence $(\xx_j)_{j=1}^\infty$ with $n_j\le \min(\supp(\xx_j))$ for all $j\in\NN$ is equivalent to the canonical $\ell_{p}$-basis (see e.g.\ \cite{AADK2016}*{Theorems 2.2 and 3.1}). Then the result follows from Theorem~\ref{thm:AnsoSP}.
\end{proof}

\section{The main theorem}\label{sect:main}
\noindent Our main result is Theorem~\ref{thm:main} below, which establishes some easy-to-check conditions that suffice to guarantee that a pair $(\XX, \UU)$ is splitting for unconditional basis, where $\XX$ is a quasi-Banach space and $\UU$ is a Banach space. Note that all the information we use about $\XX$ is obtained exclusively through its Banach envelope!

\begin{theorem}\label{thm:main}
Let $\XX$ be a quasi-Banach space and let $\UU$ be a subprojective Banach space. Suppose that $\widehat{\XX}$ and $\UU$ have unconditional bases which split into unconditional bases $(\BB_j)_{j\in A}$ and $(\BB_j)_{j\in B}$ respectively, where $(A,B)$ is a partition of $\{1,\dots,n\}$ for some $n\in\NN$. Suppose that for $1\le j\le n$ the basis $\BB_j$ satisfies an upper $r_j$-estimate and a lower $q_j$-estimate, where $q_j$ and $r_j$ are both in $[1,\infty]$, and that
\begin{enumerate}[label=(\roman*),leftmargin=*,widest=iii]
\item $q_j < r_{j+1}$ for all $1\le j \le n-1$;
\item $r_j>1$ for all $j\in B$; and
\item either $\XX$ is natural or $q_j<\infty$ for all $j\in B$.
\end{enumerate}
Then, $(\XX,\UU)$ is splitting for unconditional bases. In particular, if $\XX$ and $\UU$ have a UTAP unconditional basis, then $\XX\oplus\UU$ has a UTAP unconditional basis.
\end{theorem}

Prior to tackling the proof Theorem~\ref{thm:main}, we gather some results that we will need.

\begin{theorem}\label{thm:distant}
Let $\UU$ be a Banach lattice that satisfies an upper $r$-estimate for some $r>1$. Then $\UU$ is not a proper envelope of natural spaces. If, in addition, $\UU$ satisfies a lower $q$-estimate for some $q<\infty$, then $\UU$ is not a proper envelope space.
\end{theorem}

\begin{proof}
By \cite{LinTza1977II}*{Propositions 1.d.4 and 1.f.3}, the space $\UU^*$ has cotype $q$ for some $q<\infty$. Then, \cite{Kalton1986}*{Theorem 3.4} yields the desired result. To prove the second part of the theorem, we use \cite{LinTza1977II}*{Theorem 1.f.10} to see that $\UU$ has type $p$ for some $p>1$ followed by \cite{Kalton1986}*{Theorem 1.1}.
\end{proof}

\begin{theorem}[see \cite{Ortynski1979}*{Theorem 4}]\label{thm:Ortynski}
Let $\XX$ and $\YY$ be quasi-Banach spaces. Suppose that every operator from $\XX$ into $\YY$ is strictly singular. Then $(\XX,\YY)$ is splitting for complemented subspaces.
\end{theorem}

Now we prove a result that is of interest by itself in the theory because it allows to obtain new examples of Banach spaces with a unique unconditional basis.
\begin{theorem}\label{thm:split}
Let $(\XX_j)_{j=1}^n$ be a finite family of Banach spaces, each of which has an unconditional basis $\BB_j$. Suppose there are sequences $(q_j)_{j=1}^{n-1}$ and $(r_j)_{j=2}^{n}$, both in $[1,\infty]$, such that:
\begin{itemize}
\item $\BB_j$ satisfies a lower $q_j$-estimate for all $1\le j\le n-1$;

\item $\BB_j$ satisfies a an upper $r_j$-estimate for all $2\le j\le n$; and
\item $q_j<r_{j+1}$ for all $1\le j \le n-1$.
\end{itemize}
Then $(\XX_j)_{j=1}^n$ is splitting for unconditional bases.
\end{theorem}

\begin{proof}
Given $1\le s \le n-1$, the unconditional basis $\bigoplus_{j=1}^{s} \BB_j$ of $\bigoplus_{j=1}^{s} \XX_j$ satisfies a lower $q_{s}$-estimate. Thus the result follows by induction combining \cite{DLMR2000}*{Theorem 1} with \cite{Woj1978}*{Theorem 2.1}.
\end{proof}

\begin{lemma}\label{lem:111}
Let $\XX$ and $\YY$ be quasi-Banach spaces with unconditional bases $\BB_x$ and $\BB_y$, respectively. Suppose that there is $0<q\le\infty$ such that $\BB_x$ satisfies a lower $q$-estimate and $\BB_y$ satisfies an upper $q$-estimate. Then every sequence $\BB$ in $\XX\oplus \YY$ disjointly supported with respect to $\BB_x\oplus\BB_y$ has a subsequence which is equivalent either to a sequence in $\XX$ disjointly supported with respect to $\BB_x$, or to a sequence in $\YY$ disjointly supported with respect to $\BB_y$.
\end{lemma}

\begin{proof}
Put $\BB=(\uu_n,\vv_n)_{n=1}^\infty$. Passing to a subsequence we can assume that either $\inf_n \Vert \uu_n \Vert>0$ or $\Vert \uu_n \Vert\le \varepsilon_n$ for all $n\in\NN$, where $(\varepsilon_n)_{n=1}^\infty$ is a given sequence of positive scalars. In the former case, $\BB_u=(\uu_n)_{n=1}^\infty$ dominates $\EE[\ell_q]$ so that $\BB_u$ also dominates $\BB_v$. Hence, $\BB$ is equivalent to $\BB_u$. In the latter case, the principle of small perturbations yields that $\BB$ is equivalent to $\BB_v$ for a suitable choice of $(\varepsilon_n)_{n=1}^\infty$.
\end{proof}

\begin{proposition}\label{prop:CCCI}
For $n\in\NN$ let $(\XX_j)_{j=1}^n$ be a family of quasi-Banach spaces, each of which has unconditional basis $\BB_j$. Suppose there are sequences $(q_j)_{j=1}^{n-1}$ and
$(r_j)_{j=2}^{n}$, both in $(0,\infty]$, such that:

\begin{itemize}
\item $q_j<r_{j+1}$ for all $1\le j \le n-1$;

\item $\BB_j$ satisfies a lower $q_{j}$-estimate for all $1\le j \le n-1$; and

\item $\BB_j$ satisfies an upper $r_{j}$-estimate for all $2\le j \le n$.
\end{itemize}
Then, if $(A,B)$ is a partition of $\{1,\dots,n\}$ the spaces $\XX_a:=\bigoplus_{j\in A}\XX_j$ and $\XX_b:=\bigoplus_{j\in B}\XX_j$ are totally incomparable.
\end{proposition}

\begin{proof}
Assume by contradiction that there is an infinite-dimensional quasi-Banach space $\YY$ isomorphic to a subspace of both $\XX_a$ and $\XX_b$. By Lemma~\ref{lem:BSQB}, passing to a subspace we can suppose that $\YY$ has a normalized Schauder basis $\BB$ equivalent both to a sequence finitely disjointly supported with respect to $\bigoplus_{j\in A} \BB_j$ and to a sequence finitely disjointly supported with respect to $\bigoplus_{j\in B} \BB_j$. By Lemma~\ref{lem:111}, passing to a subbasis we obtain $i\in A$ and $k\in B$ such that $\BB$ is equivalent both to a sequence finitely disjointly supported with respect to $\BB_i$ and to a sequence finitely disjointly supported with respect to $\BB_k$. Switching the roles of $A$ and $B$ if necessary, we assume that $i<k$. We infer that $\EE[\ell_{r_{k}}]$ dominates $\EE[\ell_{q_{i}}]$ so that $r_{k}\le q_{i}$. Since $q_i<r_{i+1}\le r_k$ we reach an absurdity.
\end{proof}

We are now ready to patch together all the different pieces that play a part in the proof of our main theorem.

\begin{proof}[Completion of the Proof of Theorem~\ref{thm:main}]
By Theorem~\ref{thm:distant}, either $\UU$ is not a proper envelope, or $\XX$ is natural and $\UU$ is not a proper envelope of natural spaces. By Theorem~\ref{thm:split}, the pair $(\widehat{\XX},\UU)$ is splitting for unconditional bases. By Proposition~\ref{prop:CCCI}, $\widehat{\XX}$ and $\UU$ are incomparable. By Lemma~\ref{lem:CIfromEnv}, $\XX$ and $\UU$ are semi-complementably incomparable. By Corollary~\ref{cor:CItoSS}, every operator from $\XX$ to $\UU$ is strictly singular. By Theorem~\ref{thm:Ortynski}, $(\XX,\UU)$ is splitting for complemented subspaces. Applying Lemma~\ref{lem:split} the proof is over.
\end{proof}

\subsection{Applications}

Although we stated Theorem~\ref{thm:main} in all generality that our techniques permitted, here we will only apply it in the following cases:
\begin{itemize}[leftmargin=*]
\item $A=\{1\}$, $B=\{2\}$, and there are $1\le q_1<r_2\le q_2<\infty$ such that $\BB_1$ is $q_1$-concave, and $\BB_2$ is $r_2$-convex and $q_2$-concave;
\item $A=\{1,3\}$, $B=\{2\}$, and there are $1\le q_1<r_2\le q_2<r_3$ such that $\BB_1$ is $q_1$-concave, $\BB_2$ is $r_2$-convex and $q_2$-concave, and $\BB_3$ is $r_3$-convex.
\end{itemize}

Notwithstanding, Theorem~\ref{thm:main} is crucial in order to obtain new examples of spaces with a unique unconditional basis up to a permutation. We refer the reader to \cite{AlbiacAnsorena2020b}*{Corollary 6.2} for a comprehensive inventory of spaces $\XX$ to which Theorem~\ref{thm:main} is relevant. In fact, if $\XX$ is a direct sum built as explained in \cite{AlbiacAnsorena2020b}*{Corollary 6.2}, then $\XX$ is an $L$-convex lattice; moreover $\widehat{\XX}$ is either a $q$-concave lattice for any $q>1$ or $r$-convex lattice for any $r<\infty$, or a direct sum of both. In light of Proposition~\ref{prop:SPDS}, Theorem~\ref{thm:TsirelsonSP}, and \cite{RuizSanchez2018}*{Theorem 5.1}, Theorem~\ref{thm:main} applies to a direct sum $\UU$ built with $r$-convexifications of the Tsirelson space for $r>1$, duals of these convexified spaces, and Nakano spaces associated to a sequence $(p_j)_{j=1}^\infty$ with $\inf_j p_j>1$ and $\sup_j p_j<\infty$.

Because of their importance in Analysis, we single out some examples involving Hardy spaces. For the convenience of the reader we will next state a few known facts about the spaces $H_p(\TT^d)$ that we will need in order to apply Theorem \ref{thm:main}. The first unconditional bases in $H_p(\TT)$ for $0<p<1$ were constructed in \cite{Woj1984}. Those bases allow a manageable expression for the norm in terms of the coefficients relative to the basis. Namely, if $(\xx_n)_{n=0}^\infty$ is such a normalized basis then
\begin{equation}\label{dyadic}
\left\|\sum_{n=0}^\infty a_n \, \xx_n\right\|_{H_p(\TT)}\approx \left( \int_0^1\left(\sum_{n=0}^\infty|a_n|^2 h_n^2\right)^{p/2}\right)^{1/p}, \ (a_n)_{n=1}^\infty\in c_{00},
\end{equation}
where $(h_n)_{n=0}^\infty$ is the classical Haar system on $[0,1]$ normalized with respect to the norm in $L_p([0,1])$. Using (\ref{dyadic}) one can easily see that $(\xx_n)_{n=0}^\infty$ has a block basic sequence equivalent to the unit vector basis of $\ell_2$.

Those bases allow tensor constructions of unconditional bases in $H_p(\TT^d)$ for $d\in \NN$ which satisfy an equivalence analogous to \eqref{dyadic}. Using those tensored bases, Kalton et al.\ \cite{KLW1990} showed that the spaces $H_p(\TT^d) $ and $H_p(\TT^m)$ with $0<p<1$ and $d$, $m\in \NN$, are isomorphic if and only if $d=m$. Then it was proved in \cite{Woj1997} that all the spaces $H_p(\TT^d)$ for $0<p<1$ and $d\in \NN$ have a UTAP unconditional basis. Also from the $d$-dimensional version of \eqref{dyadic} we conclude that the (unique) unconditional basis of $H_p(\TT^d)$ has a block basic sequence equivalent to the unit vector basis in $\ell_2$.  One can also show that the Banach envelope of  $H_p(\TT^d)$ is isomorphic to $\ell_1$.

\begin{theorem}\label{newaditions}
Let $\XX$ be the finite direct sum of  some of the spaces from the following list:
\begin{itemize}[leftmargin=*]
\item The Hardy space $H_p(\TT^d)$ for $0<p<1$ and $d\in\NN$;
\item The Nakano space $\ell(p_{n})$, where $(p_n)_{n=1}^\infty$ is a non-increasing sequence in $[1,\infty)$ with $\lim_n p_n=1$ and $\sup_n (p_n-p_{2n})\log(1+n)<\infty$;
\item The Nakano space $h(q_{n})$, where $(q_n)_{n=1}^\infty$ is a non-decreasing sequence in $[1,\infty)$ with $
\lim_n q_n=\infty$ and
$\sup_n({1/q_n} - {1/q_{2n}})\log(1+n)<\infty$;
\item Tsirelson's space $\Ts$;
\item the original Tsirelson's space $\Ts^*$.
\end{itemize} Let $\YY$ be one of the spaces $\ell_2$, $\Ts^{(2)}$, or $\Ts^{(2)}_*$. Then the space $\XX\oplus\YY$ has a UTAP unconditional basis.
\end{theorem}

As the alert reader might have noticed, all known Banach spaces with a UTAP unconditional basis that follow pattern \ref{closec0} are $r$-convex lattices for $r<\infty$, all known Banach spaces with a UTAP unconditional basis that follow pattern \ref{closel1} are $q$-concave lattices for $q>1$, and all known Banach spaces with a UTAP unconditional basis that follow pattern \ref{closel2} are both $q$-convex and $r$-concave lattices for $q<2<r$. Thus Theorem~\ref{newaditions} yields in particular  new additions to the list of Banach spaces with a UTAP unconditional basis.

The main questions that Theorem~\ref{thm:split} leave open in the spirit of the \emph{Memoir} by Bourgain et al.\ \cite{BCLT1985} are whether $\ell_1(c_0)\oplus\ell_2$, $c_0(\ell_1)\oplus\ell_2$, $c_0(\ell_2)\oplus \ell_1(\ell_2)$ and $\ell_2\oplus \Ts^{(2)}$ have a UTAP unconditional basis. 

\begin{remark}
We would like to point out that trying to generalize, first \cite{Woj1978}*{Theorem 2.1} and then Theorem~\ref{thm:split} to quasi-Banach spaces, is a priori a feasible program to tackle Question~\ref{question:gluel2} in the case when $\XX$ is non-locally convex. However, we quickly run into quasi-Banach spaces, such as $H_p(\TT)$ for $0<p<1$, with a UTAP unconditional basis which, despite following pattern \ref{BEclosel1}, contain a block basic sequence equivalent to the unit vector system of $\ell_2$. Thus, in particular they do not satisfy a lower $q$-estimate for any $q<2$ and so we would not be able to apply the wished-for generalization of Theorem~\ref{thm:split} to them. This is the reason why in this paper we drew a route to approach Question~\ref{question:gluel2} based on the (necessarily incomplete) information that we get from the envelopes.
\end{remark}

\begin{remark}
Although here we are mainly concerned with and motivated by pattern \ref{closel2}, it is worth it noting that our methods are more general. As an example let us look at the space $H_p(\TT^d)\oplus c_0$. It was proved in \cite{AlbiacAnsorena2020b} that it has UTAP unconditional basis. However, the result now also easily follows from Theorem \ref{thm:main}.
\end{remark}



\end{document}